\documentclass[12pt]{amsart}
\usepackage[mathscr]{eucal}
\usepackage{amssymb}
\usepackage{latexsym}
\usepackage{amsthm}
\usepackage{amsmath}
\usepackage{graphicx}

\theoremstyle{plain}

\newtheorem{thm}{Theorem}[section]
\newtheorem{Cor}{Corollary}[section]
\newtheorem{lem}{Lemma}[section]

\theoremstyle{definition}

\title{Density measures and additive property}
\author{Ryoichi Kunisada}

\address{Department of Mathematical Science, Graduate School of Science and Engineering, Waseda University, Shinjuku-ku, Tokyo 169-8555, Japan}

\email{tk-waseda@ruri.waseda.jp}
\date{}

\begin{document}
\maketitle

\begin{abstract}
We consider a certain class of normalized positive linear functionals on $l^{\infty}$ which extend the Ces\`{a}ro mean. We study the set of its extreme points and it turns out to be the set of linear functionals constructed from free
ultrafilters on natural numbers $\mathbb{N}$. Also, regarding them as finitely additive measures defined on all subsets of $\mathbb{N}$, which are often called density measures, we study a certain additivity property of such measures being equivalent to the completeness of the $L^p$-spaces on such measures. Particularly a necessary and sufficient condition for such a density measure to have this property is obtained.
\end{abstract}

\bigskip

\section{Introduction}
Let us $l^{\infty}$ be the Banach space of all real-valued bounded functions on $\mathbb{N}$ and $L^{\infty}(\mathbb{R}_+^{\times})$ be the Banach space of all real-valued essentially bounded measureable functions on $\mathbb{R}_+^{\times} = [1, \infty)$. Let ${(l^{\infty})}^*$ and ${L^{\infty}(\mathbb{R}_+^{\times})}^*$ be their conjugate spaces respectively.
 The symbol $\mathcal{P}(\mathbb{N})$ stands for the family of all subsets of $\mathbb{N}$, and for a set $A \in \mathcal{P}(\mathbb{N})$ let $|A|$ denote the cardinality of $A$.
Recall that the Ces\`{a}ro mean of a function $f \in \l^{\infty}$ is defined as
\[C(f) = \lim_{n \to \infty} \frac{1}{n}\sum_{i=1}^n f(i) \]
if this limit exists. When $f$ is the characteristic function $I_A$ of a set $A \in \mathcal{P}(\mathbb{N})$, its Ces\`{a}ro mean $C(I_A) = D(A) =  \lim_{n \to \infty} \frac{|A \cap [1, n]|}{n}$ is often called asymptotic density of $A$.
We consider a class $\mathcal{C}$ of normalized positive linear functionals on $l^{\infty}$ concerning Ces\`{a}ro summability method. Namely, linear functionals $\varphi$ on $l^{\infty}$ satisfying the following condition:
\[\varphi(f) \le \overline{C}(f) = \limsup_{n \to \infty} \frac{1}{n} \sum_{i=1}^{n} f(i) \]
for each $f \in \l^{\infty}$. It is remarked that such a functional $\varphi$ is an extension of Ces\`{a}ro mean, that is, $\varphi(f) = C(f)$ provided the limit exists. $\mathcal{C}$ is a compact convex set in its weak* topology and hence by the Krein-Milman theorem, the set of extreme points $ex(\mathcal{C})$ of $\mathcal{C}$ is not an empty set. An example of such a functional is given by
\[ \varphi^{\mathcal{U}}(f) = \mathcal{U}\mathchar`-\lim_n \frac{1}{n} \sum_{i=1}^n f(i), \]
where $f \in l^{\infty}$ and the limit in the above definition means the limit along an free ultrafilter $\mathcal{U}$ on $\mathbb{N}$ (The precise definition of this notion is given in the following section). We denote the set of all such functionals by $\tilde{\mathcal{C}}$. Remark that there are distinct free ultrafilters $\mathcal{U}$ and $\mathcal{U}^{\prime}$ which give the same element of $\mathcal{C}$, thus $\tilde{\mathcal{C}}$ is isomorphic as a set to some quotient space of the set of free ultrafilters on $\mathbb{N}$. We show that each functional in $\tilde{\mathcal{C}}$ is equal to some $\varphi^{\mathcal{U}}$ with $\mathcal{U}$ is a certain kind of ultrafilter, which has a form convenient to investigate the associated functional. The relation between $\tilde{\mathcal{C}}$ and $\mathcal{C}$ can be understood simply in view of the theory of linear topological spaces, that is,  
$\tilde{\mathcal{C}}$ is precisely the set of extreme points $ex(\mathcal{C})$ of $\mathcal{C}$. Also we show that each element of $\mathcal{C}$ can be expressed as an integral with respect to some unique probability measure supported by its extreme points. It may be regarded as an interesting example concerning Choquet's theorem. For these purposes, it is useful to introduce an integral analogy $\mathcal{M}$ of $\mathcal{C}$ which is a class of normalized positive linear functionals on $L^{\infty}(\mathbb{R}_+^{\times})$ defined by using the subadditive functional $\overline{M}$ on $L^{\infty}(\mathbb{R}_+^{\times})$ which adopts the integral with respect to the Haar measure of real line $\mathbb{R}$ in place of the summation: namely, $\mathcal{M}$ is the set of linear functionals $\psi$ on $L^{\infty}(\mathbb{R}_+^{\times})$ for which
\[\psi(f) \le \overline{M}(f) = \limsup_{x \to \infty} \frac{1}{x} \int_1^x f(t) dt \]
holds for every $f \in L^{\infty}(\mathbb{R}_+^{\times})$. Similarly we define a subclass $\tilde{\mathcal{M}}$ of $\mathcal{M}$ consisting of those $\psi^{\mathcal{U}}$ defined by
\[ \psi^{\mathcal{U}}(f) = \mathcal{U}\mathchar`-\lim_x \frac{1}{x} \int_1^x f(t) dt, \]
where $\mathcal{U}$ is a ultrafilter on $\mathbb{R}_+^{\times}$ which contains no bounded set of $\mathbb{R}_+^{\times}$ and again the limit means the limit along $\mathcal{U}$. In fact, it turns out that $\mathcal{C}$ and $\mathcal{M}$ are isomorphic as a compact convex set and that definitions and results obtained in the integral setting can be transferred to the summation setting with ease. Therefore in Section 3 and 4, where we study these problems, we mainly work with the integral setting in which arguments are simpler. 
 
 Furthermore, in the integral setting, we can naturally consider the continuous flow on $\tilde{\mathcal{M}}$ induced by the action of the multiplicative group $\mathbb{R}^{\times} = (0, \infty)$ of positive real numbers $\mathbb{R}_+$ on $\mathbb{R}_+^{\times}$ defined as follows: let us consider a semiflow on $\mathbb{R}_+^{\times}$ as follows.
\[\rho^s : \mathbb{R}_+^{\times} \longrightarrow \mathbb{R}_+^{\times}, \quad \rho^s x = 
2^s x, \quad s \ge 0. \]
Then define linear operators $P_s$ as
\[ P_s : L^{\infty}(\mathbb{R}_+^{\times}) \longrightarrow L^{\infty}(\mathbb{R}_+^{\times}),
\quad (P_sf)(x) = f(\rho^sx), \quad s \ge0. \]
Let $P_s^*$ be the adjoint operators of $P_s$, then
\[P_s^* : \tilde{\mathcal{M}} \longrightarrow \tilde{\mathcal{M}}, \quad s \ge 0 \]
are homeomorphisms and $(\tilde{\mathcal{M}}, \{P_s^*\}_{s \in \mathbb{R}})$ is a continuous flow (The proof of this fact is given in Section 3). This flow plays a fundamental role in the last section.

Another subject of this paper is concerned with the notion of finitely additive measures. Recall that finitely additive measures defined on $\mathcal{P}(\mathbb{N})$ which extend the asymptotic density are called density measures. Notice that $\mathcal{C}$ can be considered to be a subclass of density measures when we recognize them as finitely additive measures defined on $\mathcal{P}(\mathbb{N})$, that is, restricting $\varphi \in \mathcal{C}$ to characteristic functions of sets in $\mathcal{P}(\mathbb{N})$, we obviously get a density measure $\nu$. In particular, we denote the corresponding density measure of $\varphi^{\mathcal{U}}$ by $\nu^{\mathcal{U}}$. Density measures have been studied by several authors from various points of view, see for instance [4, 5, 8, 9, 10, 12, 14]. Following [4], [6] and [10] we will deal with a certain additivity property of density measures, which is, roughly speaking, a weakening of countable additivity. This kind of property was studied firstly by Buck [6] and Mekler [10], who called it the additive property. The authors of [4] have also studied that property and a natural weakening of it. In this paper we shall deal with the latter alone, the definition of which is given generally as follows. Let $(X, \mathcal{F}, \mu)$ be a finitely additive finite measure space where $\mathcal{F}$ is a $\sigma$-algebra of subsets of $X$. We say that $\mu$ has the additive property if
for any increasing sequence $\{A_i\}_{i = 1}^{\infty}$ of $\mathcal{F}$, there exists a set $B \in \mathcal{F}$ such that 

\quad (1) $\mu(B) = \lim_{i \to \infty}  \mu(A_i)$,

\quad (2) $\mu(A_i \setminus B) = 0 \quad for \ every \ i = 1,2, \cdots$. 
\medskip

  In what follows, the \textit{additive property} will always mean this one. This property is tightly linked to the notion of $L^p$ spaces over finitely additive measures (see [3] for details). It is known that $L^1(\mu)$ is complete if and only if $\mu$ has the additive property (see for instance [1] and [2]).
 It was shown in [4, Theorem 1] that there exists a density measure with the additive property; namely, if a free ultrafilter $\mathcal{U}$ on $\mathbb{N}$ contains a set $\{n_k\}_{k=1}^{\infty}$
such that 
\[\lim_{k \rightarrow \infty} \frac{n_{k+1}}{n_k} = \infty \]
then the density measure $\nu^{\mathcal{U}}$ has the additive property. We shall generalize the result and prove a necessary and sufficient condition for density measures in $\tilde{\mathcal{C}}$ to have the additive property.

The paper is organized as follows: In the next section we introduce necessary notions and notation which will be used throughout the paper. Also we present some basic results on density measures, including the fact that the class $\mathcal{C}$ is a proper subset of the set of density measures. In Section 3 after proving the affine homeomorphism between $\mathcal{C}$ and $\mathcal{M}$, we show that each element of $\tilde{\mathcal{C}}$ has a special expression, which fact induces the topological structure of $\tilde{\mathcal{C}}$ that $\tilde{\mathcal{C}}$ is homeomorphic to a closed subset of the maximal ideal space of the space of all uniformly continuous bounded functions on $[1, \infty)$. 

In Section 4, we show the result that $\tilde{\mathcal{C}} = ex(\mathcal{C})$. Although it is relatively easy to show that $ex(\mathcal{C}) \subseteq \tilde{\mathcal{C}}$ by applying the Krein-Milman theorem, it is rather difficult to prove that $\tilde{\mathcal{C}}$ is exactly $ex(\mathcal{C})$ and we will have to prepare some amount of machinery. After that we show the representation theorem for general elements of $\mathcal{C}$.

Section 5 is devoted to the study of the additivity property of density measures in $\tilde{\mathcal{C}}$. Recurrence property of elements of $\tilde{\mathcal{C}}$ for the flow defined above will be used to characterize their additive property. Applying our results we shall show an example of a density measure in $\tilde{\mathcal{C}}$ which has the additive property but does not satisfy the above condition.

\section{Preliminaries}
 In the sequel, each measure is supposed to be a finitely additive probability measure defined on $\mathcal{P}(\mathbb{N})$.   Generally, giving a measure $\mu$, one can define a normalized positive linear functional $\varphi$ on $l^{\infty}$ in a similar way to the definition of Lebesgue integral. Conversely, take any normalized positive linear functional $\varphi$ on $l^{\infty}$, then we can obtain the measure $\mu$ by putting $\mu(A) = \varphi(I_A)$ for every $A \in \mathcal{P}(\mathbb{N})$. Therefore we can identify these two notions by this correspondence.

 Now let us consider the relation between $\mathcal{C}$ and the class of density measures. As we have mentioned above, to each density measure, there corresponds a normalized positive linear functional on $l^{\infty}$. It is shown in [8] that functionals corresponding to density measures are precisely the positive functionals extending 
Ces\`{a}ro mean. We denote the set of all such functionals by $\mathcal{P}$, which is clearly a weak* compact convex subset of $(l^{\infty})^*$. Then the following result is known [8, Proposition 5.5]:
\[\overline{P}(f) = \sup_{\varphi \in \mathcal{P}} \varphi(f) = \lim_{\theta \to 1-} \limsup_{n \to \infty} \frac{\sum_{i \in [\theta n, n]} f(i)} {n - \theta n} \]
for each $f \in l^{\infty}$. This functional $\overline{P}$ is an extension of  P\'{o}lya density for bounded sequences.

 Since $\overline{C}(f) \le \overline{P}(f)$ for every $f \in l^{\infty}$, it is obvious that $\mathcal{C} \subseteq \mathcal{P}$. And it is known that there exists a element $f$ of $l^{\infty}$ such that $\overline{C}(f) < \overline{P}(f)$ (for example, see [6, P. 572]), so we have that $\mathcal{C} \varsubsetneqq \mathcal{P}$. 

 To construct a linear functional which assigns a value to all bounded sequences we need the notion of the limit along an ultrafilter. We give below the definition in a general setting. Let $X$ be a set. Let $\mathcal{U}$ be a ultrafilter on $X$ and let $f : X \rightarrow \mathbb{R}$ be any bounded function. Then there exists a unique real number $\alpha$ such that for any $\varepsilon > 0$, $\{x \in X : |f(x) - \alpha| < \varepsilon \} \in \mathcal{U}$. In this case we write
\[\mathcal{U}\mathchar`-\lim_x f(x) = \alpha \]
and say that the number $\alpha$ is the $\mathcal{U}$-limit of $f$.
 
 We consider $\tilde{\mathcal{C}}$ as a topological space endowed with the relative topology of the weak* topology of $(l^{\infty})^*$. From this point of view, it is convenient to use the notion of the Stone-\v{C}ech compactification $\beta\mathbb{N}$ of $\mathbb{N}$. As is well known, the space of all ultrafilters on $\mathbb{N}$ can be identified with $\beta\mathbb{N}$. In particular, a free ultrafilter corresponds to a point in $\beta\mathbb{N} \setminus \mathbb{N}$, here which is denoted by $\mathbb{N}^*$. The algebra of clopen subsets of $\mathbb{N}^*$ form a topological basis for $\mathbb{N}^*$ and which are precisely the sets $A^* = \overline{A} \cap \mathbb{N}^*$ for each $A \in \mathcal{P}(\mathbb{N})$, where $\overline{A}$ denotes the closure of a set $A \in \mathcal{P}(\mathbb{N})$ in $\beta\mathbb{N}$.
 Recall that $l^{\infty}$ is isometric to $C(\beta\mathbb{N})$, the space of all real-valued continuous functions on $\beta\mathbb{N}$. For each $f \in l^{\infty}$, the isomorphic image $\overline{f}$ in $C(\beta\mathbb{N})$ is given by its continuous extension to $\beta\mathbb{N}$: namely,
\[\overline{f}(\mathcal{U}) = \mathcal{U}\mathchar`-\lim_n f(n) \]
for every ultrafilter $\mathcal{U}$ on $\mathbb{N}$, i.e., for every point $\mathcal{U}$ in $\beta\mathbb{N}$.

 Another notion pertaining to $\mathbb{N}^*$ which is important for our study is an extension of right translation on $\mathbb{N}$. We define a mapping $\tau_0 : \mathbb{N} \rightarrow \mathbb{N}$ by $\tau_0(n) = n+1$. Regarding it as a mapping from $\mathbb{N}$ to $\beta\mathbb{N}$, we can extend it to a continuous mapping on $\beta\mathbb{N}$. We denote this extension by $\tau$. The restriction of $\tau$ to $\mathbb{N}^*$ is a homeomorphism of $\mathbb{N}^*$ onto itself and we denote it by the same symbol $\tau$ as well. Then $(\mathbb{N}^*, \tau)$ is a topological dynamics.

Further, we consider the continuous flow $(\Omega^*, \{\tau^s\}_{s \in \mathbb{R}})$ of the suspension of the discrete flow $(\mathbb{N}^*, \tau)$, whose construction is well known in topological dynamics (for example see [15, Chapter 2]) and is given as follows. Let us consider a product space $\beta\mathbb{N} \times [0,1]$ and construct the compact space $\Omega$ by identifying all the pairs of points $(\eta, 1)$ and $(\tau \eta, 0)$ for all $\eta \in \beta\mathbb{N}$. Also we denote by $\Omega^*$ the closed subspace of $\Omega$ consisting of all elements $(\eta, t)$ in $\Omega$ with $\eta  \in \mathbb{N}^*$. Then we define a continuous flow on $\Omega^*$ extending $(\mathbb{N}^*, \tau)$ as follows; for each $s \in \mathbb{R}$, we define the homeomorphism $\tau^s : \Omega^* \rightarrow \Omega^*$ by
\[\tau^s(\eta, t) = (\tau^{[t+s]}\eta, t+s-[t+s]), \]
where $[x]$ denotes the largest integer not exceeding $x$ for a real number $x$. We shall use this flow in Section 5.

\section{The topological structure of the space $\tilde{\mathcal{C}}$}
 In this section we will investigate details of the compact Hausdorff space $\tilde{\mathcal{C}}$ defined in the former section. The main purpose of this section is to prove the following result, which was suggested by arguments in the proof of [5, Lemma 5]. In what follows, we denote a general element of $\beta\mathbb{N}$ by $\eta$ and those of $\Omega$ by $\omega$.
 
\begin{thm}
Each element of $\tilde{\mathcal{C}}$ can be expressed uniquely in the form
\[ \varphi_{\omega} (f) = \eta\mathchar`-\lim_n \frac{1}{\theta \cdot 2^n} \sum_{i=1}^{[\theta \cdot 2^n]} f(i) \]
for some $\omega = (\eta, t)$ in $\Omega^*$, where $\theta = 2^t$. Also this correspondence of $\Omega^*$ to $\tilde{\mathcal{C}}$ is continuous, that is, $\tilde{\mathcal{C}}$ is homeomorphic to $\Omega^*$.
\end{thm}
 
This result plays an important role in proving our theorems in Section 5 and is interesting in its own right. It is helpful to introduce the notion of the image of an ultrafilter to understand the above limit. Let $X$ and $Y$ be arbitrary sets and given a mapping $f : X \rightarrow Y$. For any ultrafilter $\mathcal{U}$ on X, one can define the ultrafilter on Y, denoted by $f(\mathcal{U})$ consisting of those $A \subseteq Y$ for which $f^{-1}(A) \in \mathcal{U}$. Then it is easy to see that
\[f(\mathcal{U})\mathchar`-\lim_y g(y) = \mathcal{U}\mathchar`-\lim_x g \circ f(x), \]
where g is any bounded function on $Y$.

 Let us $\mathbb{R}_+ = [0, \infty)$ and $\mathbb{R}_+^{\times} = [1, \infty)$.  We particularly consider the following three maps; $\mathbb{R}_+ \ni x \mapsto 2^x \in \mathbb{R}_+^{\times}$, $\mathbb{R}_+^{\times} \ni x \mapsto [x] \in \mathbb{N}$, $\mathbb{R}_+^{\times} \ni x \mapsto \theta x \in \mathbb{R}_+^{\times}$, where $\theta \ge 1$. We denote the images of ultrafilter $\mathcal{U}$ of the induced mappings defined above by $2^{\mathcal{U}}, [\mathcal{U}], \theta \mathcal{U}$, respectively. Notice that $2^{\mathcal{U}}$ is a ultrafilter on $\mathbb{R}_+^{\times}$ which does not contain any bounded set of $\mathbb{R}_+^{\times}$ if and only if $\mathcal{U}$ is a ultrafilter on $\mathbb{R}_+$ of the same kind, and those can be considered to be equal, then the map $\mathcal{U} \rightarrow 2^{\mathcal{U}}$ is a bijection of the set of all such ultrafilters on $\mathbb{R}_+^{\times}$ onto itself. Notice that with the notation above we can write $\varphi_{\omega} = \varphi^{[2^{\omega}]} = \varphi^{[\theta 2^{\eta}]}$. 

We will need some more preparation to prove the theorem. Let $C_{ub}(\mathbb{R}_+^{\times})$ be the space of all real-valued  uniformly continuous bounded functions on $\mathbb{R}_+^{\times}$. Its maximal ideal space, denoted here by $\mathfrak{M}$, is a compact Hausdorff space and the space $C(\mathfrak{M})$ of all real-valued continuous functions on $\mathfrak{M}$ is isometric to $C_{ub}(\mathbb{R}_+^{\times})$ as an Banach algebra. The following lemma is a consequence of [13, Lemma 2.1], but we give here a proof, for the sake of completeness:
 
\begin{lem}
$\mathfrak{M}$ is homeomorphic to $\Omega$.
\end{lem}

\begin{proof}
 It is sufficient to show the algebraic isomorphism $C_{ub}(\mathbb{R}_+^{\times}) \cong C(\Omega)$. If we regard the points $(n, t)$ in $\Omega$ with $n \in \mathbb{N}$ as the points $n+t$ in $\mathbb{R}_+^{\times}$ we can consider that $\Omega$ contains $\mathbb{R}_+^{\times}$ as a dense subspace, so that $\Omega$ is a compactification of $\mathbb{R}_+^{\times}$. Now given any $f \in C_{ub}(\mathbb{R}_+^{\times})$, put $f_n(s) = f(n+s), \ s \in [0,1], n=1,2, \cdots$. Then we have a sequence $\{f_n\}_{n=1}^{\infty}$ of $C([0,1])$. Since $f$ is bounded and uniformly continuous on $\mathbb{R}_+^{\times}$, it follows that this sequence is uniformly bounded and equicontinuous. Hence by Arzel\`{a}-Ascoli's theorem, $\{f_n\}_{n=1}^{\infty}$ is relatively compact in $C([0,1])$ in its uniform topology. Therefore when we put
\[\Phi_f : \mathbb{N} \longrightarrow C([0,1]), \ \Phi_f(n) = f_n, n=1,2, \cdots, \]
then we can extend it continuously to $\beta\mathbb{N}$. Then we define a continuous function $\overline{f}$ on $\Omega$ by 
\[\overline{f}(\omega) = (\Phi_f(\eta))(t), \quad (\eta \in \beta\mathbb{N}, t \in [0,1]). \]
We denote this mapping $f \mapsto \overline{f}$ by $\Phi : C_{ub}(\mathbb{R}_+^{\times}) \rightarrow C(\Omega)$. Notice that $f = \overline{f}$ on $\mathbb{R}_+^{\times}$, so that $\overline{f}$ is a continuous extension of $f$ to $\Omega$. In particular, it is obvious that $\Phi$ is injective. We shall show that $\Phi$ is a algebraic isomorphism. It is trivial that $\Phi$ is a algebraic homomorphism. To show that $\Phi$ is surjective, it is sufficient to show that for every continuous function $g$ on $\Omega$ its restriction to $\mathbb{R}_+^{\times}$ is uniformly continuous on $\mathbb{R}_+^{\times}$. Now we regard $g$ as a mapping from $\beta\mathbb{N}$ to $C([0,1])$ with uniform topology:
\[\Phi_g : \beta\mathbb{N} \longrightarrow C([0,1]), \quad \Phi_g(\omega) = g(\omega, t), \]
then $\Phi_g$ is continuous. Since $\Phi_g(\beta\mathbb{N})$ is a compact subset of $C([0,1])$,
$\Phi_g(\mathbb{N})$ is relatively compact in $C([0,1])$. Hence $\{\Phi_g(n)\}_{n=1}^{\infty} = \{g(n+t)\}_{n=1}^{\infty}$ is equicontinuous. Thus $g$ is uniformly continuous on $\mathbb{R}_+^{\times}$.
\end{proof}

Thus we can identify $\mathfrak{M}$ with $\Omega$, so that in the sequel we will use only the symbol $\Omega$. 
Notice that $\Omega$ is the compactification of $\mathbb{R}_+^{\times}$ to which any uniformly continuous bounded function $f(x)$ on $\mathbb{R}_+^{\times}$ can be extended continuously. In particular, we can see from the above proof that, for any $f \in C_{ub}(\mathbb{R}_+^{\times})$ and $\omega = (\eta, t) \in \Omega$, its continuous extension $\overline{f}(\omega)$ is given by the formula
\[\overline{f}(\omega) = \omega\mathchar`-\lim_s f(s), \]
where $\omega$ is regarded as an ultrafilter on $\mathbb{R}_+^{\times}$ generated by the basis $\{A + t : A \in \eta \}$. From now on, we often identify a point $\omega = (\eta, t) \in \Omega$ with the above ultrafilter. An immediate consequence of these facts which will be used in the next section is that for any cluster point $\alpha$ of the set $\{f(x)\}_{x \in \mathbb{R}_+^{\times}}$, there exists a point $\omega \in \Omega$ such that $\overline{f}(\omega) = \alpha$. Since we are mainly interested in the extended values of $f(x) \in C_{ub}(\mathbb{R}_+^{\times})$, that is, cluster points of
$\{f(x)\}_{x \ge 1}$ as $x \to \infty$, we may often ignore the difference in values on bounded sets of $\mathbb{R}_+^{\times}$ among members in $C_{ub}(\mathbb{R}_+^{\times})$; namely, we consider a member of $C_{ub}(\mathbb{R}_+^{\times})$ modulo $C_0(\mathbb{R}_+^{\times})$, where $C_0(\mathbb{R}_+^{\times})$ is the ideal of $C_{ub}(\mathbb{R}_+^{\times})$ consisting of all those members $f(x)$ which converges to zero as $x$ tends to $\infty$. Then it holds that
\[ C(\Omega^*) = C_{ub}(\mathbb{R}_+^{\times}) / C_0(\mathbb{R}_+^{\times}), \]
where $C(\Omega^*)$ is the space of all real-valued continuous functions on $\Omega^*$.

In what follows, we shall show an affine homeomorphism between $\mathcal{C}$ and $\mathcal{M}$ and then introduce a version of Theorem 3.1 which is formulated in the integral setting. For each $f \in \l^{\infty}$, we define a function $\tilde{f} \in L^{\infty}(\mathbb{R}_+^{\times})$ by  $\tilde{f}(x) = f([x])$. Then we define an affine continuous mapping $V$ as follows:
\[V : \mathcal{M} \longrightarrow \mathcal{C}, \quad (V\psi)(f) = \psi(\tilde{f}). \]

\begin{thm}
$V$ is a affine homeomorphism between $\mathcal{C}$ and $\mathcal{M}$.
\end{thm}

\begin{proof}
First we show that $V$ is surjective. It is noted that for each $f \in l^{\infty}$
\[\frac{1}{n} \sum_{i=1}^n f(i) = \frac{1}{n} \int_1^{n+1} \tilde{f}(t)dt. \] 
Let us $\tilde{l}^{\infty} = \{\tilde{f}(x) \in L^{\infty}(\mathbb{R}_+) : f \in l^{\infty}\}$. Given any $\varphi \in \mathcal{C}$, we define a functional $\psi_0$ on $\tilde{l}^{\infty}$ by $\psi_0(\tilde{f}) = \varphi(f)$ for every $f \in l^{\infty}$. Since
\[ \psi_0(\tilde{f}) = \varphi(f) \le \limsup_n \frac{1}{n} \sum_{i=1}^n f(i) = \limsup_x \frac{1}{x}
\int_1^x \tilde{f}(t)dt \]
holds from above, we can extend $\psi_0$ to $\psi \in \mathcal{M}$ by the Hahn-Banach theorem. Then we have obviously that $V(\psi) = \varphi$, which shows that $V$ is surjective. Next we show that $V$ is injective. It is sufficient to show that for any $f \in L^{\infty}(\mathbb{R}_+^{\times})$, there exists a function $g \in l^{\infty}$ such that $\psi(f) = \psi(\tilde{g})$ for every $\psi \in \mathcal{M}$. In fact, suppose that this holds and let $\psi, \psi_1$ be two distinct elements of $\tilde{\mathcal{M}}$ with $V\psi = V\psi_1$. Then there is some $f \in L^{\infty}(\mathbb{R}_+^{\times})$ such that $\psi(f) \not= \psi_1(f)$. On the other hand, there exists some $g \in l^{\infty}$ such that $\psi(f) = \psi(\tilde{g}) = (V\psi)(g), \psi_1(f) = \psi_1(\tilde{g}) = (V\psi_1)(g)$, i.e., $\psi(f) = \psi_1(f)$, which is a contradiction. We can get such a function $g(n)$ simply by putting $g(n) = \int_n^{n+1} f(t)dt, n=1,2, \cdots$. Therefore we have shown that $V$ is an affine homeomorphism.
\end{proof}

For any $\omega \in \Omega^*$ we define $\psi_{\omega} =\psi^{2^{\omega}}$, i.e.,
\[\psi_{\omega}(f) = 2^{ \omega}\mathchar`-\lim_x \frac{1}{x} \int_1^x f(t)dt = \omega\mathchar`-\lim_x \frac{1}{2^x} \int_1^{2^x} f(t)dt. \]
We denote this mapping of $\Omega^*$ to $\tilde{\mathcal{M}}, \ \omega \mapsto \psi_{\omega}$ by $\Psi$. The following lemma is obvious.
\begin{lem}
$V$ maps $\tilde{\mathcal{M}}$ onto $\tilde{\mathcal{C}}$ and $V\psi_{\omega} = \varphi_{\omega}$ holds for every $\omega \in \Omega^*$.
\end{lem}

From this lemma, Theorem 3.1 is equivalent to the assertion that $\Psi$ is a homeomorphism, which we will prove sequentially.
For the sake of simplicity, we will use a linear operator $U : L^{\infty}(\mathbb{R}_+^{\times}) \longrightarrow L^{\infty}(\mathbb{R}_+^{\times})$ defined as $Uf(x) = \frac{1}{x}\int_1^x f(t)dt$, and can write that $\psi^{\mathcal{U}} (f) = \mathcal{U}\mathchar`-\lim_x (Uf)(x)$. Also let us define the linear operator $W$ as follows:
\[W : L^{\infty}(\mathbb{R}_+^{\times}) \longrightarrow L^{\infty}(\mathbb{R}_+), \quad (Wf)(x) = f(2^x). \]
First, we will need the following elementary lemma.

\begin{lem}
If $f \in L^{\infty}(\mathbb{R}_+^{\times})$, then $WUf \in C_{ub}(\mathbb{R}_+^{\times})$.
\end{lem}

\begin{proof}
Let $f$ be in $L^{\infty}(\mathbb{R}_+^{\times})$ and $h$ be a positive real number, then we have
\[(Uf)(x+h) - (Uf)(x) = -\frac{h}{x+h}(Uf)(x) + \frac{1}{x+h}\int_x^{x+h} f(t)dt. \]
Hence we get that
\[|(Uf)(x+h) - (Uf)(x)| \le \frac{2h\|f\|_{\infty}}{x+h}. \]
Let $s, \theta \in \mathbb{R}_+$ and put $x = 2^s, \ h = 2^{s+\theta} - 2^s$. Applying above results, we have
\[|(Uf)(2^{s+\theta}) - (Uf)(2^s)| \le \frac{2\cdot2^s(2^{\theta}-1)\|f\|_{\infty}}{2^{s+\theta}} = 2\|f\|_{\infty}(1-\frac{1}{2^{\theta}}). \]
The right hand side of the equation tends to 0 monotonically as $\theta \rightarrow 0$, and that does not depend on $s$. Then $(WUf)(s)$ is uniformly continuous on $\mathbb{R}_+^{\times}$.
\end{proof}

Notice that by the above result it can be written as $\psi_{\omega}(f) = \omega\mathchar`-\lim_x (WUf)(x) = \overline{(WUf)}(\omega)$. 

\begin{lem}
$\Psi$ is continuous.
\end{lem}

\begin{proof}
Let $\{\omega_{\alpha}\}_{\alpha \in \Lambda}$ be a net in $\Omega^*$ which converges to $\omega$. We will show that
\[\lim_{\alpha} \psi_{\omega_{\alpha}}(f) = \psi_{\omega}(f) \]
for every $f \in L^{\infty}(\mathbb{R}_+^{\times})$. From the assumption, we have that for any
$g \in C(\Omega^*)$ 
\[\lim_{\alpha} g(\omega_{\alpha}) = g(\omega). \]
Notice that $\overline{WUf}$ is in $C(\Omega^*)$ and then we have
\[\lim_{\alpha} \overline{WUf}(\omega_{\alpha}) = \overline{WUf}(\omega), \]
which implies that
\[\lim_{\alpha} \psi_{\omega_{\alpha}}(f) = \psi_{\omega}(f). \]
The proof is complete.
\end{proof}

\begin{lem}
$\Psi$ is surjective.
\end{lem}

\begin{proof}
We take any $\psi^{\mathcal{U}} \in \tilde{\mathcal{M}}$. Then we shall show that there exists a point $\omega = (\eta, t) \in \Omega^*$ such that $\psi_{\omega} = \psi^{\mathcal{U}}$. As we have mentioned before, since the mapping $\mathcal{U} \mapsto 2^{\mathcal{U}}$ is a bijection of the set of ultrafilters on $\mathbb{R}_+^{\times}$ not containing any bounded set of $\mathbb{R}_+^{\times}$ onto itself, we can get the inverse $\mathcal{U}_0$ of $\mathcal{U}$, that is, $\mathcal{U} = 2^{\mathcal{U}_0}$. Then it follows that
\[\psi^{\mathcal{U}}(f) = \mathcal{U}\mathchar`-\lim_x (Uf)(x) = 2^{\mathcal{U}_0}\mathchar`-\lim_x (Uf)(x) = \mathcal{U}_0\mathchar`-\lim_x (WUf)(x) \]
for every $f \in L^{\infty}(\mathbb{R}_+^{\times})$. Since $(WUf)(x) \in C_{ub}(\mathbb{R}_+^{\times})$, $\mathcal{U}_0$ can be replaced by some $\omega = (\eta, t) \in \Omega^*$. Therefore we have that
\[\psi^{\mathcal{U}}(f) = \omega\mathchar`-\lim_x (WUf)(x) = 2^{\omega}\mathchar`-\lim_x (Uf)(x) = \psi_{\omega}(f). \] 
The proof is complete.
\end{proof}

\begin{lem}
$\Psi$ is injective.
\end{lem}

\begin{proof}
It is sufficient to show that for any pair $\omega, \omega^{\prime}$ of distinct elements of $\Omega^*$, there exists a set $X \in \mathcal{B}(\mathbb{R}_+^{\times})$ such that $\psi_{\omega}(I_X) \not= \psi_{\omega^{\prime}}(I_X)$, where $\mathcal{B}(\mathbb{R}_+^{\times})$ denotes the set of Borel subsets of $\mathbb{R}_+^{\times}$ and $I_X$ denotes the characteristic function of $X$. We divide the proof into two cases according to whether one is contained in the orbit of the other or not. Let us denote $o(\omega) = \{\tau^s \omega : s \in \mathbb{R} \}$, the orbit of $\omega$ under $\{\tau^s\}_{s \in \mathbb{R}}$. \\
\underline{case 1}. $\omega^{\prime} \in o(\omega)$. \\
Without loss of generality, we can assume that $\omega^{\prime} = \tau^s\omega (s > 0)$. Let $\omega = (\eta, t)$. We take a
set $A \in \eta$ such that $|n-m| \ge [s] +2$ whenever $n,m \in A, n \not= m$. Then we 
define a set $X \in \mathcal{B}(\mathbb{R}_+^{\times})$ as $X = \cup_{n \in A} (2^{t+n-1}, 2^{t+n}]$. We will show that $\psi_{\omega}(I_X) \not= \psi_{\omega^{\prime}}(I_X)$.
Now assume oppositely that $\psi_{\omega}(I_X) = \psi_{\omega^{\prime}}(I_X) = 
\alpha$. Let $\varepsilon$ be a positive number with $\varepsilon < \frac{1-2^{-s}}{1+2^{-s}} \alpha$. Then there exists a set $B \in \eta$ such that $B \subseteq A$ and 
\[\Big|\frac{1}{2^{t+x}} \int_1^{2^{t+x}} I_X(y) dy - \alpha \Big| < \varepsilon \quad and \quad 
\Big|\frac{1}{2^{s+t+x}} \int_1^{2^{s+t+x}} I_X(y) dy - \alpha \Big| < \varepsilon \]
whenever $x \in B$. Observing that by the assumption of $A$, $X \cap (2^{t+n}, 2^{s+t+n}] = \emptyset$ for any $n \in A$. We have then that if $x \in B$,
\begin{align}
\int_1^{2^{t+x}} I_X(y)dy < 2^{t+x}(\alpha + \varepsilon) &\Longrightarrow \int_1^{2^{s+t+x}} 
I_X(y)dy < 2^{t+x}(\alpha + \varepsilon) \notag \\
&\Longleftrightarrow \frac{1}{2^{s+t+x}} \int_1^{2^{s+t+x}} I_X(y)dy \le 2^{-s}(\alpha + \varepsilon) < \alpha - \varepsilon, \notag
\end{align}
which is a contradiction. \\
\underline{case2}. $\omega^{\prime} \notin o(\omega)$. \\
Let us $\omega = (\eta, t)$ and $\omega^{\prime} = (\eta^{\prime}, t^{\prime})$. We take $A \in \eta$ such that $\tau^{-1} A \cup A \cup \tau A  \cup \tau^2 A \notin \eta^{\prime}$. We set $X = \cup_{n \in A} (2^{t+n-1}, 2^{t+n}]$, then it is obvious that 
\begin{align}
\psi_{\omega}(I_X) &= \omega\mathchar`-\lim_x \frac{1}{2^{t+x}} \int_1^{2^{t+x}} I_X(y)dy \notag \\
&\ge \liminf_{x \in A} \frac{1}{2^{t+x}} \int_1^{2^{t+x}} I_X(y)dy \notag \\
&\ge \frac{2^{t+x} - 2^{t+x-1}}{2^{t+x}} = \frac{1}{2}. \notag
\end{align}
Hence in order to show that $\psi_{\omega}(I_X) \not= \psi_{\omega^{\prime}}(I_X)$, 
it is sufficient to show that $\psi_{\omega^{\prime}}(I_X) < \frac{1}{2}$. Now we choose $B \in \eta^{\prime}$ such that $(\tau^{-1} A \cup A \cup \tau A \cup \tau^2 A) \cap B = \emptyset$. Then for any $x^{\prime} \in B$ we have $(2^{t^{\prime} + x^{\prime} -2}, 2^{t^{\prime} + x^{\prime}}] \cap X = \emptyset$. Then we have that
\[\psi_{\omega^{\prime}}(I_X) \le \limsup_{x^{\prime} \in B} \frac{1}{2^{t^{\prime} + x^{\prime}}} \int_1^{2^{t^{\prime} + x^{\prime}}} I_X(y)dy \le \frac{2^{t^{\prime} + x^{\prime} -2}}{2^{t^{\prime} + x^{\prime}}} = \frac{1}{4}, \]
which proves the theorem.
\end{proof}
Therefore, since $\Psi$ is a continuous bijective mapping from $\Omega^*$ to $\tilde{\mathcal{M}}$, it is a homeomorphism. We have completed the proof of Theorem 3.1.

Now we explain that the mapping $\Psi : \Omega^* \rightarrow \tilde{\mathcal{M}}$ carry over the structure of continuous flow on $\Omega^*$ defined in Section 2 into the continuous flow on $\tilde{\mathcal{M}}$ defined in Section 1; for any $r = 2^s$ with $s \ge 0$ we have that

\begin{align}
(P_s^* \psi_{\omega})(f) = \psi_{\omega}(P_s f) &= 2^{\omega}\mathchar`-\lim_x \frac{1}{x} \int_1^x f(rt)dt \notag \\
&= 2^{\omega}\mathchar`-\lim_x \frac{1}{rx} \int_r^{rx} f(t)dt  \notag \\
&= r \cdot 2^{\omega}\mathchar`-\lim_x \frac{1}{x} \int_1^x f(t)dt \notag \\
&= 2^{\tau^s \omega}\mathchar`-\lim_x \frac{1}{x} \int_1^x f(t)dt. \notag \\
&= \psi_{\tau^s \omega}(f). \notag
\end{align}
Therefore, we have obtained the following result, which asserts that the two continuous flows $(\Omega^*, \{\tau^s\}_{s \in \mathbb{R}})$ and $(\tilde{\mathcal{M}}, \{P_s^*\}_{s \in \mathbb{R}})$ are isomorphic via $\Psi$.
\begin{thm}
$\Psi \circ \tau^s = P_s^* \circ \Psi$ holds for each $s \in \mathbb{R}$.
\end{thm}

\section{The set of extreme points of $\mathcal{C}$}
 In this section we investigate the algebraic structure of $\mathcal{C}$ as a compact convex set. For this purpose, it is equivalent to study $\mathcal{M}$ by Theorem 3.2. So here we continue to work with $\mathcal{M}$. First we begin with the following relatively elementary result.

\begin{thm}
 Let $ex(\mathcal{M})$ be the set of all extreme points of $\mathcal{M}$. Then $\tilde{\mathcal{M}}$ is weak* compact, and $ex(\mathcal{M}) \subseteq \tilde{\mathcal{M}}$ holds.
\end{thm}

\begin{proof}
The compactness of $\tilde{\mathcal{M}}$ follows from Lemma 3.4. By the Krein-Milman theorem and Lemma 3.5, it is sufficient to show that
\[\sup_{\omega \in \Omega^*} \psi_{\omega}(f) = \overline{M}(f) \]
for every $f \in L^{\infty}(\mathbb{R}_+^{\times})$. It is obvious that
\[ \overline{{M}}(f) = \limsup_{x \to \infty} (Uf)(x) = \limsup_{x \to \infty} (WUf)(x) = \overline{(WUf)}(\omega) 
 = 2^{\omega}\mathchar`-\lim_x \frac{1}{x} \int_1^x f(t)dt = \psi_{\omega}(f). \]
for some $\omega \in \Omega^*$. We are done.
\end{proof}

Let us $P(\Omega^*)$ be the set of all probability Borel measures on $\Omega^*$. Given any $\mu \in P(\Omega^*)$, the integral defined as follows yields a member $\psi$ of $\mathcal{M}$;
\[ \psi(f) = \int_{\Omega^*} \psi_{\omega}(f) d\mu(\omega). \]
Since for any $\omega \in \Omega^*$ the measure $\delta_{\omega} \in P(\Omega^*)$, the probability measure equals 1 on any Borel subset of $\Omega^*$ which contains $\omega$ and equals 0 otherwise, induces $\psi_{\omega} \in \mathcal{M}$, we can regard this mapping of $P(\Omega^*)$ to $\mathcal{M}$ as a extension of $\Psi$ and we denote it by the symbol $\overline{\Psi}$. Now Theorem 4.1 together with the Krein-Milman theorem asserts that this mapping $\overline{\Psi}$ is surjective ([11, Section 1]):
\begin{Cor}
Every member $\psi$ of $\mathcal{M}$ can be expressed in the form
\[ \psi(f) = \int_{\Omega^*} \psi_{\omega}(f) d\mu(\omega). \]
for some probability measure $\mu$ on $\Omega^*$.
\end{Cor}

 In connection with the linear operator $U$, we introduce a subspace $\mathfrak{U}$ of $C_{ub}(\mathbb{R}_+^{\times})$ as follows:
\[ \mathfrak{U} = \{f(x) \in C_{ub}(\mathbb{R}_+^{\times}) : (xf(x))^{\prime} \in L^{\infty}(\mathbb{R}_+^{\times}) \}. \]
In other words, $f(x) \in C_{ub}(\mathbb{R}_+^{\times})$ is in $\mathfrak{U}$ if and only if
derivative $(xf(x))^{\prime}$ exists almost everywhere on $\mathbb{R}_+^{\times}$ and also
it is an essentially bounded measurable function. 

\begin{lem}
$\mathfrak{U}$ is a subalgebra of $C_{ub}(\mathbb{R}_+^{\times})$.
\end{lem}

\begin{proof}
We show that it is closed under multiplication. For $f(x) \in \mathfrak{U}$, notice that $(xf(x))^{\prime} = f(x) + xf^{\prime}(x) \in L^{\infty}(\mathbb{R}_+^{\times})$ and which implies that 
$xf^{\prime}(x) \in L^{\infty}(\mathbb{R}_+^{\times})$ since $f(x)$ is bounded on $\mathbb{R}_+^{\times}$. Now let us given arbitrary pair of elements $f, g$ of $\mathfrak{U}$.
Then we get by the product rule,
\[(x(fg)(x))^{\prime} = f(x)g(x) + xf^{\prime}(x)g(x) + xg^{\prime}(x)f(x) \]
exists almost everywhere on $\mathbb{R}_+^{\times}$ and it is essentially bounded since as mentioned above, $xf^{\prime}(x)$ and $xg^{\prime}(x)$ are bounded. Hence $fg$ is in $\mathfrak{U}$.
\end{proof}
Now we take up the relation between $\mathfrak{U}$ and the range $UL^{\infty}$ of the operator $U$. A hat placed above the symbol for a subalgebra of $C_{ub}(\mathbb{R}_+^{\times})$ will be used to indicate its quotient algebra modulo the ideal $C_0(\mathbb{R}_+^{\times})$; for example, $\hat{\mathfrak{U}} = \mathfrak{U} / (\mathfrak{U} \cap C_0(\mathbb{R}_+^{\times}))$. Let us take any $f(x)$ in $\mathfrak{U}$ and put $(xf(x))^{\prime} = \xi_f(x)$. Then
\begin{align}
& xf(x) - f(1) = \int_1^x \xi_f(t)dt, \quad x \ge 1 \notag \\
&\Longleftrightarrow f(x) =\frac{1}{x} \int_1^x \xi_f(t)dt + \frac{f(1)}{x} = (U\xi_f)(x) + \frac{f(1)}{x}, \quad x \ge 1. \notag
\end{align}
Hence,
\[ f(x) \equiv (U\xi_f)(x) \pmod{C_0(\mathbb{R}_+^{\times})}. \]
Conversely, for any $f \in L^{\infty}(\mathbb{R}_+^{\times})$ it is obvious that $Uf \in \mathfrak{U}$. This leads to the following lemma.

\begin{lem}
$\hat{\mathfrak{U}} = \widehat{(UL^{\infty})}$.
\end{lem}

The next result is essential to prove our main theorem.
\begin{lem}
$\widehat{(WUL^{\infty})}$ is uniformly dense in $C(\Omega^*)$.
\end{lem}

\begin{proof}
We use the Stone-Weierstrass theorem to prove the theorem. First, $1 \in \widehat{(WUL^{\infty})}$ is obvious. The fact that $\widehat{(WUL^{\infty})}$ separates points in $\Omega^*$
follows from Lemma 3.6. Finally, by Lemma 4.1 and 4.2, $\widehat{(UL^{\infty})}$ is an algebra and hence 
$\widehat{(WUL^{\infty})}$ is also an algebra. Then we can apply the Stone-Weierstrass theorem and
get the result.
\end{proof}

With the aid of these results, we now prove our main theorem.
\begin{thm}
$\tilde{\mathcal{M}} = ex(\mathcal{M})$.
\end{thm}

\begin{proof}
Since we have already shown that $ex(\mathcal{M}) \subseteq \tilde{\mathcal{M}}$ in Theorem 4.1, we have to prove only that $\tilde{\mathcal{M}} \subseteq ex(\mathcal{M})$. 
Let us assume that for some $\omega \in \Omega^*$ and some $\psi_1, \psi_2 \in \mathcal{M}$,
\[ \psi_{\omega} = \alpha \psi_1 + (1-\alpha) \psi_2, \quad 0 < \alpha < 1. \]
By Corollary 4.1, there exist probability measures $\mu, \nu$ on $\Omega^*$ such that
\[ \psi_1(f) = \int_{\Omega^*} \psi_{\omega^{\prime}}(f) d\mu(\omega^{\prime}), \quad
\psi_2(f) = \int_{\Omega^*} \psi_{\omega^{\prime}}(f) d\nu(\omega^{\prime}) \]
for every $f \in L^{\infty}(\mathbb{R}_+^{\times})$. Then if we put $\lambda = \alpha \mu + (1-\alpha) \nu \in P(\Omega^*)$, we have that
\begin{align}
& \psi_{\omega}(f) = \alpha \psi_1 + (1-\alpha) \psi_2 = \int_{\Omega^*}  \psi_{\omega^{\prime}}(f) d\lambda(\omega^{\prime}) \quad for \ every \ f \in L^{\infty}(\mathbb{R}_+^{\times}) \notag \\
&\Longleftrightarrow \overline{(WUf)}(\omega) = \int_{\Omega^*} \overline{(WUf)}(\omega^{\prime})
d\lambda(\omega^{\prime}) \quad for \ every \ f \in L^{\infty}(\mathbb{R}_+^{\times}) \notag \\
&\Longleftrightarrow \overline{g}(\omega) = \int_{\Omega^*} \overline{g}(\omega^{\prime}) d\lambda(\omega^{\prime}) \quad for \ every \ g \in WUL^{\infty} \notag \\
&\Longrightarrow h(\omega) = \int_{\Omega^*} h(\omega^{\prime}) d\lambda(\omega^{\prime}) \quad for \ every \ h \in C(\Omega^*). \notag
\end{align}
Therefore, $\lambda = \delta_{\omega}$ holds. Thus the support sets of $\mu$ and $\nu$ are $\{\omega\}$ and we conclude that $\mu = \nu = \delta_{\omega}$, that is, $\psi_1 = \psi_2 = \psi_{\omega}$. This completes the proof.
\end{proof}

Next theorem is an immediate consequence of the proof of Theorem 4.2.
\begin{thm}
For any $\psi \in \mathcal{M}$, a probability measure $\mu$ on $\Omega^*$ which represents $\psi$ is unique. Namely, $\psi$ is uniquely expressed in the form
\[ \psi(f) = \int_{\Omega^*} \psi_{\omega}(f) d\mu(\omega). \]
for some $\mu \in P(\Omega^*)$.
\end{thm}

 A consequence of Theorem 4.3 is that $\overline{\Psi} : P(\Omega^*) \rightarrow \mathcal{M}$ is an affine homeomorphism. Also the isomorphism $\Psi$ between the two flows $(\Omega^*, \{\tau^s\}_{s \in \mathbb{R}})$ and $(\tilde{\mathcal{M}}, \{P_s^*\}_{s \in \mathbb{R}})$ established in Section 3 can be extended to an isomorphism between their closed convex hulls; we define linear operators $T_s$ in a similar way as $P_s$ 
for each $s \ge 0$.
\[ T_s : C_{ub}(\mathbb{R}_+^{\times}) \longrightarrow C_{ub}(\mathbb{R}_+^{\times}), \quad
(T_s f)(x) = f(x+s), \quad s \ge 0. \]
Let $T_s^*$ be their adjoint operators. Then we have that the two continuous flows $(P(\Omega^*), \{T_s^*\}_{s \in \mathbb{R}})$ and $(\mathcal{M}, \{P_s^*\}_{s \in \mathbb{R}})$ are isomorphic via $\Psi$:

\begin{thm}
$\Psi \circ T_s^* = P_s^* \circ \Psi$ holds for each $s \in \mathbb{R}$.
\end{thm}

We give below our main results of this section formulated in the summation setting.

\begin{Cor}
$ex(\mathcal{C}) = \tilde{\mathcal{C}}$.
\end{Cor}

\begin{Cor}
For any $\varphi \in \mathcal{C}$, there is a unique probability measure $\mu$ on $\Omega^*$ such that
\[ \varphi(f) = \int_{\Omega^*} \varphi_{\omega}(f) d\mu(\omega) \]
holds for every $f \in l^{\infty}$.
\end{Cor}

\section{The additive property of density measures in $\tilde{\mathcal{C}}$}
In this section we study the additive property of elements of $\tilde{\mathcal{C}}$. Recall that we say that $\mu \in \tilde{\mathcal{C}}$ has the additive property if for any increasing sequence $A_1 \subset A_2 \subset \cdots \subset A_k \subset \cdots$ of $\mathcal{P}(\mathbb{N})$, there exists a set $B \subseteq \mathbb{N}$ such that

\quad (1) $\mu(B) = \lim_k \mu(A_k)$,

\quad (2) $\mu(A_k \setminus B) = 0 \quad for \ every \ k \in \mathbb{N}$. 
\medskip

 Now we need to prepare some notions relative to the topological dynamics $(\mathbb{N}^*, \tau)$ and the continuous flow $(\Omega^*, \{\tau^s\}_{s \in \mathbb{R}})$ before stating our main theorem. Recall that a point $\eta \in \mathbb{N}^*$ is called wondering for $\tau$ if there is an open neighborhood $U$ of $\eta$ such that the sets $\tau^{n}U$, $n$ is any integers, are mutually disjoint. We denote the set of all wondering points by $\mathcal{W}_d$. Also we denote by $\mathcal{D}_d$ the subset of $\mathbb{N}^*$ consisting of all points that does not return arbitrarily close to the initial point under negative iteration by $\tau$(i.e., $\eta \in \mathbb{N}^*$ is in $\mathcal{D}_d$ if and only if there exists a open neighborhood $U$ of $\eta$ such that $U \cap \{\tau^{-n}\eta : n \ge 1\} = \emptyset)$. This is equivalent to the condition that  the orbit $\{\tau^{-n} \eta : n \ge 0\}$ is a discrete space in its relative topology. $\mathcal{W}_d \subseteq \mathcal{D}_d$ is clear by the definitions.

 Next we define similar notions for the continuous flow $(\Omega^*, \{\tau^s\}_{s \in \mathbb{R}})$.
For a point $\omega \in \Omega^*$, $\omega$ is called wondering if there are open neighborhoods $U$ of $\omega$ and $V$ of $0 \in \mathbb{R}$ such that $U \cap \tau^s U = \emptyset$ for every s in $\mathbb{R} \setminus V$. We denote the set of all wondering points by $\mathcal{W}$. Similarly we denote by $\mathcal{D}$ the subset of $\Omega^*$ consisting of all points for which there are open neighborhoods $U$ of $\omega$ and $V$ of $0 \in \mathbb{R}$ such that $\tau^{-s} \omega \not\in U$ for every $s \in \mathbb{R}_+ \setminus V$. Note that this is equivalent to the condition that the orbit $\{\tau^{-s}\omega : s \ge 0\}$ is homeomorphic to $\mathbb{R}_+$ in its relative topology. It is easy by the definitions to check  that $\omega = (\eta, t) \in \mathcal{W}$ if and only if $\eta \in \mathcal{W}_d$, and $\omega = (\eta, t) \in \mathcal{D}$ if and only if $\eta \in \mathcal{D}_d$. In particular $\mathcal{W} \subseteq \mathcal{D}$ holds.

For simplicity we will use the symbol $|A \cap n| = |A \cap [1, n]|$ in the following proof. Recall that by Theorem 3.1, each density measure $\nu^{\mathcal{U}}$ in $\tilde{\mathcal{C}}$ is equal to $\nu_{\omega}$ for some $\omega = (\eta, t) \in \Omega^*$ defined as follows:
\[\nu_{\omega}(A) = \eta\mathchar`-\lim_n \frac{|A \cap [\theta \cdot 2^n]|}{\theta \cdot 2^n}, \quad A \in \mathcal{P}(\mathbb{N}), \]
where $\theta = 2^t$. Then we have the following theorem.

\begin{thm}
$\nu_{\omega}$ has the additive property if and only if $\omega \in \mathcal{D}$.
\end{thm}

\begin{proof}
(Sufficiency)
The following proof is based on the proof of [4, Theorem 1].
 We put $\omega = (\eta, t)$ and $\theta = 2^t$. Firstly, remark that by the assumption there exists some $X \in \eta$ such that $\{\tau^{-n} \eta : n \ge 1\} \cap X^* = \emptyset$. Fix such a set X, and put $X = \{n_k\}_{k=1}^{\infty}$. Then for every $m \ge 1$, $X \setminus (X \cap \tau^m X) \in \eta$ holds. Indeed, since $\tau^m X \not \in \eta$, then $X \cap \tau^m X \not \in \eta$.
Since $\eta$ is a ultrafileter, $X \setminus (X \cap \tau^m X) \in \eta$. Now put
\[Y_i := \{n_k \in X: n_k - n_{k-1} \ge i \}, \quad i \ge 1\]
then we have 
$Y_i \in \eta$ for every $i \ge 1$. Indeed, we can easily see that
\begin{align}
Y_i &= (X \setminus (\cup_{m=1}^{i-1} (X \cap \tau^m X)) \notag \\
&= \cap_{m=1}^{i-1} (X \setminus (X \cap \tau^m X)), \quad i = 1,2, \cdots. \notag
\end{align}
Then we get the result since $X \setminus (X \cap \tau^m X) \in \eta$ for every $m \ge 1$.

 Now we shall prove that $\nu_{\omega}$ has the additive property. Take any increasing sequence $\{A_i\}_{i \ge 1}$ of $\mathcal{P}(\mathbb{N})$ and put $\alpha = \lim_i \nu_{\omega}(A_i)$. We choose a decreasing sequence $\{Z_i\}_{i \ge 1}$ of $\omega$ such that for every $i \ge 1$, $Z_i \subseteq Y_i$ and 
\[\Bigl|\frac{|A_i \cap [\theta \cdot 2^n]|}{\theta \cdot 2^n} - \nu_{\omega}(A_i)\Bigr| < \frac{1}{i} \]
whenever $n \in Z_i$. For the set $X = \{n_k\}_{k=1}^{\infty}$ we consider the partition of $\mathbb{N}$ defined as follows:
\[\mathbb{N} = \cup_{k \ge 2} I_k, \quad I_k = ([\theta \cdot 2^{n_{k-1}}], [\theta \cdot 2^{n_k}]], \ k \ge 2. \]
We then define a set $B \subseteq \mathbb{N}$ as $B \cap I_k = A_i \cap I_k$ if $n_k \in Z_i \setminus Z_{i+1}$ and $B \cap I_k = \emptyset$ if $n_k \notin Z_1$. Fix any $i \ge 1$ and take any $n_k \in Z_i$, then $n_k \in Z_j \setminus Z_{j+1}$ for some $j \ge i$. Remark that since we assumed $Z_j \subseteq Y_j$, we have $\frac{2^{n_{k-1}}}{2^{n_k}} = 2^{n_{k-1}-n_k} \le 2^{-j}$ and $|B \cap ([\theta \cdot 2^{n_{k-1}}], [\theta \cdot 2^{n_k}]]| = |A_j \cap ([\theta \cdot 2^{n_{k-1}}], [\theta \cdot 2^{n_k}]|$. It follows that
\begin{align}
\frac{|B \cap [\theta \cdot 2^{n_k}]|}{\theta \cdot 2^{n_k}} &= \frac{|B \cap ([\theta \cdot 2^{n_{k-1}}], [\theta \cdot 2^{n_k}]]|}{\theta \cdot 2^{n_k}} + \frac{|B \cap [\theta \cdot 2^{n_{k-1}}]|}{\theta \cdot 2^{n_k}} \notag \\
&\le \frac{|A_j \cap ([\theta \cdot 2^{n_{k-1}}], [\theta \cdot 2^{n_k}]]|}{\theta \cdot 2^{n_k}} 
+ \frac{1}{2^j} \notag \\
&\le \frac{|A_j \cap [\theta \cdot 2^{n_k}]|}{\theta \cdot 2^{n_k}} + \frac{1}{2^j} 
\le \nu_{\omega}(A_j) + \frac{1}{j} + \frac{1}{2^j} \le \alpha + \frac{1}{i} + \frac{1}{2^i}. \notag
\end{align}
Thus since $i \ge 1$ is arbitrary, we have
\[\nu_{\omega}(B) = \eta\mathchar`-\lim_n \frac{|B \cap [\theta \cdot 2^{n}]|}{\theta \cdot 2^{n}} \le \alpha. \]
 
 Next we shall prove that $\nu_{\omega}(A_i \setminus B) =0$ for every $i =1,2, \cdots$. Fix a number $j \ge i$. For any $n_k \in Z_j$, there exists a $l \ge j$ such that $B \cap I_k = A_l \cap I_k \supseteq A_i \cap I_k$. Thus we have  $(A_i \setminus B) \cap I_k = \emptyset$. Hence
\[\frac{|(A_i \setminus B) \cap [\theta \cdot 2^{n_k}]|}{\theta \cdot 2^{n_k}} 
=\frac{|(A_i \setminus B) \cap [\theta \cdot 2^{n_{k-1}}]|}{\theta \cdot 2^{n_k}} 
\le \frac{\theta \cdot 2^{n_{k-1}}}{\theta \cdot 2^{n_k}} \le \frac{1}{2^l} \le \frac{1}{2^j}. \]
Therefore
\[\nu_{\omega}(A_i \setminus B) = \eta\mathchar`-\lim_n \frac{|(A_i \setminus B) \cap [\theta \cdot 2^{n}]|}{\theta \cdot 2^{n}} = 0. \]
Since $\nu_{\omega}(A_i \setminus B) = 0$ implies $\nu_{\omega}(A_i) \le \nu_{\omega}(B)$, thus 
$\alpha = \lim_i \nu_{\omega}(A_i) \le \nu_{\omega}(B)$. Hence $\nu_{\omega}(B) = \alpha$. We have shown that $\nu_{\omega}$ has the additive property.

(Necessity) It is sufficient to show that if for every element X of $\eta$, there exists some $k \ge 1$ such that $\tau^k X \in \eta$, then $\nu_{\omega}$ does not have the additive property. Especially in the case of $t=0$, this is [4, Theorem 6]. It is relatively easy to modify the proof so that it works for any $t \in [0,1]$. We have done. 
\end{proof}

 In the remainder of the section, we  consider the existence of a density measure in $\tilde{\mathcal{C}}$ with the additive property. As we have seen in Theorem 5.1, there is a close relation between the additive property of density measures in $\tilde{\mathcal{C}}$ and the topological dynamics $(\mathbb{N}^*, \tau)$ or the flow $(\Omega^*, \{\tau^s\}_{s \in \mathbb{R}})$. Following Chou [7], we say a set $A \subseteq \mathbb{N}$ is thin if $A \cap \tau^n A$ is a finite set for each positive integer n. It is obvious that a point $\eta \in \mathbb{N}^*$ is in $\mathcal{W}_d$ if and only if $\omega$ is contained in the closure of a thin set $A$, that is, $\eta$ contains a thin set $A$. Chou proved $\mathcal{W}_d$ is dense in $\mathbb{N}^*$ [7, Proposition1.2]. In particular, together with our result of Theorem 5.1, the existence of a density measure $\nu_{\omega}$ having the additive property follows immediately.
\begin{lem}
For a set $A = \{n_k\}_{k=1}^{\infty}$, $A$ is a thin set if and only if
\[\liminf_{k \to \infty} (n_k -n_{k-1}) = \infty. \]
\end{lem}

\begin{proof}
Sufficiency is obvious. Suppose that $A$ is thin and
\[\liminf_{k \to \infty} (n_k -n_{k-1}) = l_A < \infty\]
then $A \cap \tau^{l_A} A$ is an infinite set. It contradicts the assumption that $A$ is thin.
\end{proof}

We give the following characterization of a density measure $\nu_{\omega}$ with $\omega$ in $\mathcal{W}$. Recall that $\nu_{\omega} = \nu^{[2^{\omega}]}$ for $\omega$ in $\Omega^*$.
\begin{thm}
For $\omega = (\eta, t)$ in $\Omega^*$, $\nu_{\omega}$ has the additive property and the associated free ultrafilter $\mathcal{U} = [2^{\omega}]$ contains a set $X=\{n_k\}_{k=1}^{\infty}$ such that
\[\lim_{k \rightarrow \infty} \frac{n_{k+1}}{n_k} = \infty \]
if and only if $\omega \in \mathcal{W}$.
\end{thm}

\begin{proof}
Note that a free ultrafilter $\mathcal{U} = [2^{\omega}] = [\theta 2^{\eta}]$ is generated by the basis $\{[\theta \cdot 2^A] : A \in \eta\}$, where $\theta = 2^t$ and $[\theta \cdot 2^A] = \{[\theta \cdot 2^n] : n \in A\}$ for each $A \in \eta$. First we prove sufficiency. Since $\omega \in \mathcal{W} \subseteq \mathcal{D}$, $\nu_{\omega}$ has the additive property by Theorem 5.1. Take any thin set $A = \{n_k\}_{k=1}^{\infty}$ contained in $\eta$, Put $X = [\theta \cdot 2^A] = \{m_k\}_{k=1}^{\infty}$, then $X \in \mathcal{U}$. By Lemma 5.1 we have that
\[\liminf_{k \to \infty} \frac{m_{k+1}}{m_k} = \liminf_{k \to \infty} \frac{[\theta \cdot 2^{n_{k+1}}]}{[\theta \cdot 2^{n_k}]} = 2^{\liminf_{k \to \infty}(n_{k+1} - n_k)} = \infty. \]

 Conversely, Assume that $\mathcal{U}$ contains a set $X = \{m_k\}_{k=1}^{\infty}$ with $\lim_{k \rightarrow \infty} \frac{m_{k+1}}{m_k} = \infty$. Since there is a set $A = \{n_k\}_{k=1}^{\infty}$ in $\eta$ such that $[\theta \cdot 2^A] \subseteq X$, then
\[\infty = \liminf_{k \to \infty} \frac{m_{k+1}}{m_k} \le \liminf_{k \to \infty} \frac{[\theta \cdot 2^{n_{k+1}}]}{[\theta \cdot 2^{n_k}]} = 2^{\liminf_{k \to \infty}(n_{k+1} -n_k)}. \]
Hence $\liminf_{k \to \infty}(n_{k+1} - n_k) = \infty$, that is, By Lemma 5.1 $A$ is a thin set. Then $\omega \in \mathcal{W}$.
\end{proof}

 In particular, this result is contained in [4, Theorem 1], which we remarked at Section 1. Then it is natural to ask that whether there exists a density measure $\nu_{\omega} \in \tilde{\mathcal{C}}$ with the additive property and the associated ultrafilter does not contain a set $\{n_k\}_{k=1}^{\infty}$ with $\lim_{k \rightarrow \infty} \frac{n_{k+1}}{n_k} = \infty$. The answer to this question is affirmative. Notice that from the above theorem, it is equivalent to $\mathcal{W} \subsetneq \mathcal{D}$ or, equivalently, $\mathcal{W}_d \subsetneq \mathcal{D}_d$.

\begin{thm}
$\mathcal{W}_d \subsetneq \mathcal{D}_d$.
\end{thm}

\begin{proof}
We put $\Gamma = \mathbb{N}^* \setminus \mathcal{W}_d$. Since $\mathcal{W}_d$ is an open invariant set, $\Gamma$ is a closed invariant subset. For any $A \subseteq \mathbb{N}$, we denote $A^* \cap \Gamma$ by $\hat{A}$. Then it is sufficient to show that there exists a set $X \subseteq \mathbb{N}$ such that
\[\hat{X} \subsetneq \cup_{i=1}^l \tau^i \hat{X} \]
for every $l \ge 1$. Indeed, if it is true, it follows that by the compactness of $\hat{X}$, $\hat{X} \setminus (\cup_{i=1}^{\infty} \tau^i \hat{X}) \not= \emptyset$, and obviously any point in the set is contained in $\mathcal{D}_d \setminus \mathcal{W}_d$.

 Take a set $X \subseteq \mathbb{N}$ and write $X = \{n_k\}_{k=1}^{\infty}$. We put
\[Y_X = \{m \in \mathbb{N} : |\{ k \ge 2 : n_k - n_{k-1} = m\}| = \infty \} \]
and notice that 
\[X \setminus (\cup_{i=1}^l \tau^i X) = \{n_i \in X : n_k -n_{k-1} > l\} \]
and
\begin{align}
\hat{X} \setminus (\cup_{i=1}^l \tau^i \hat{X}) \not= \emptyset &\Longleftrightarrow (X \setminus \cup_{i=1}^l \tau^i X) \ \widehat{} \not= \emptyset \notag \\
&\Longleftrightarrow X \setminus \cup_{i=1}^l \tau^i X \not\subseteq \mathcal{W}^{\tau} \notag \\
&\Longleftrightarrow X \setminus \cup_{i=1}^l \tau^i X \ is \ not \ a \ thin \ set. \notag
\end{align}
Hence we obtain that $\hat{X} \subsetneq \cup_{i=1}^l \tau^i \hat{X}$ for any $l \ge 1$ if and only if $ \{n_k \in X : n_k -n_{k-1} > l\}$ is not a thin set for any $l \ge 1$, i.e., $Y_X$ is an infinite set. We can see easily that such a set $X$ exists. For example, Put $X = \{2^n +2^k : n \ge 0, 0 \le k <n\}$. We have done.
\end{proof}

 Therefore for any point $\omega \in \mathcal{D} \setminus \mathcal{W}$, the density measure $\nu_{\omega}$ give an example having the additive property but does not satisfy the sufficient condition of [4, Theorem 1].
\bibliographystyle{jplain}
\bibliography{myrefs}

\end{document}